\documentclass{article}
\usepackage{maa-monthly}
\usepackage[title]{appendix}
\usepackage[cp1251]{inputenc} 

\theoremstyle{theorem}
\newtheorem{theorem}{Theorem}
\newtheorem{corollary}[theorem]{Corollary}
\newtheorem{lemma}[theorem]{Lemma}
\newtheorem{proposition}[theorem]{Proposition}
\newtheorem{conjecture}[theorem]{Conjecture}
 
\theoremstyle{definition}

\newcommand{\CC}{\mathbb{C}}
\newcommand{\RR}{\mathbb{R}}
\newcommand{\NN}{\mathbb{N}}

\DeclareMathOperator{\spn}{span}
\let\Re\undefined
\DeclareMathOperator{\Re}{Re}

\begin{document}

\title{A Positivity Conjecture Related to the Riemann Zeta Function}
\markright{A Positivity Conjecture}
\author{Hugues Bellemare, Yves Langlois and Thomas Ransford}

\maketitle

\begin{abstract}
According to two remarkable theorems of Nyman and  B\'aez-Duarte, 
the Riemann hypothesis is equivalent to a simply-stated criterion 
concerning least-squares approximation.
In carrying out computations related to this criterion, we have observed 
a curious phenomenon: 
for no apparent reason, at least the first billion entries of a certain infinite triangular matrix
 associated to the Riemann zeta function are all positive. 
In this article we describe the background leading to this observation, and 
make a conjecture. 
\end{abstract}

\section{Introduction.}\label{S:intro}

For each integer $k\ge2$, define $f_k:(0,1]\to\RR$ by
\begin{equation}
f_k(x):= \frac{1}{k}\Bigl[\frac{1}{x}\Bigr]-\Bigl[\frac{1}{kx}\Bigr]\qquad(0<x\le 1).
\end{equation}
Here $[t]$ denotes the integer part of $t$. Notice that $f_k$ is constant on each interval of the form
$(\frac{1}{j+1},\frac{1}{j}]$ and that, in this interval, $f_k(x)=\{j/k\}$, where $\{t\}:=t-[t]$ denotes the fractional
part of $t$. In particular, we have $0\le f_k(x)<1$ for all $k,x$.
The graph of  $f_5$ is displayed in Figure~\ref{F:f5}.

\begin{figure}[ht]
\begin{center}
\includegraphics[width=0.45\textwidth]{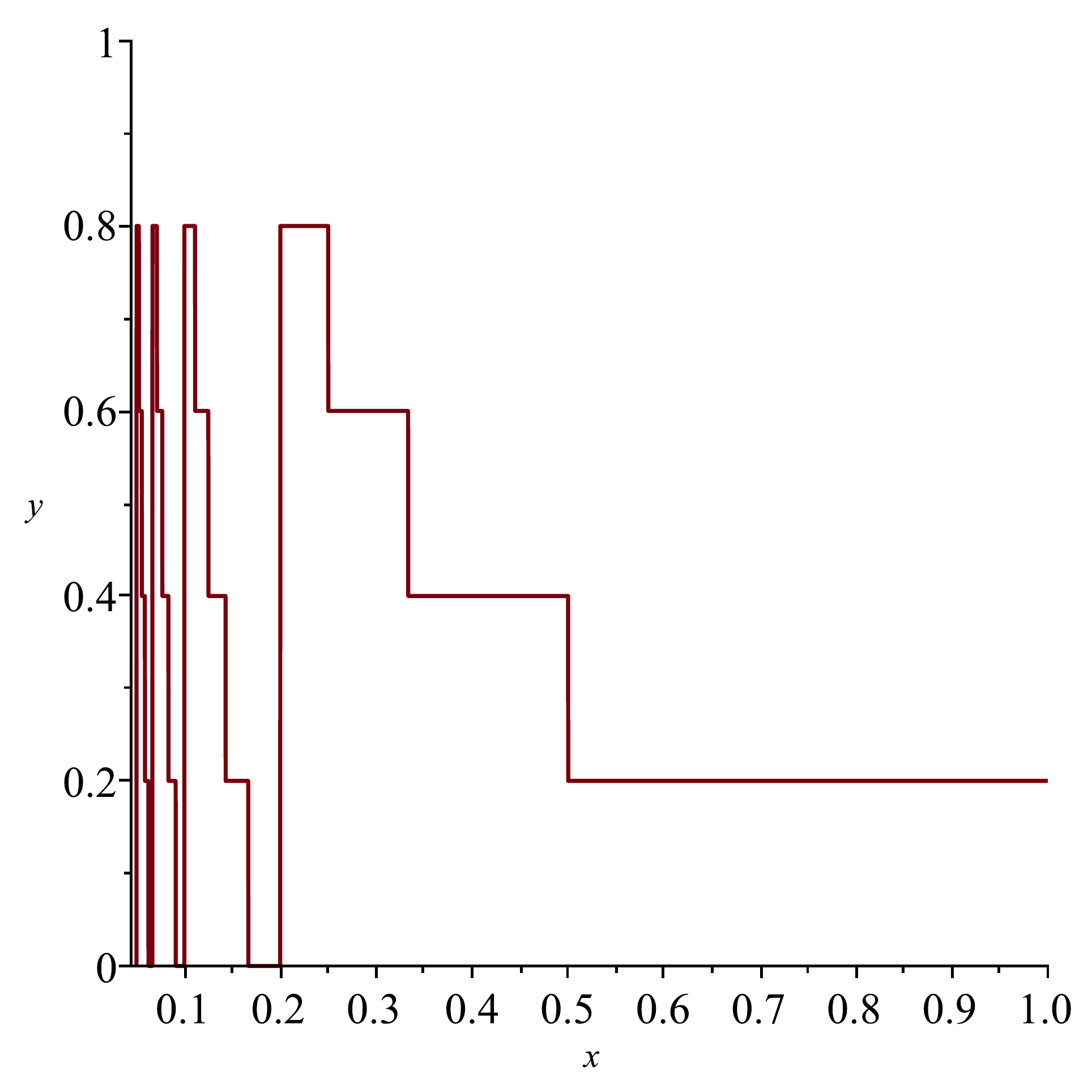}
\end{center}
\caption{Graph of $f_5(x)$.}\label{F:f5}
\end{figure}

For each integer $n\ge2$, define $d_n$ by
\begin{equation}\label{E:dn}
d_n^2:=\min\int_0^1\Bigl(1-\sum_{k=2}^n\lambda_kf_k(x)\Bigr)^2\,dx,
\end{equation}
where the minimum is taken over all real scalars $\lambda_2,\dots,\lambda_n$.
Thus $d_n$ is the distance between the constant function $1$ and the span of $\{f_2,\dots,f_n\}$,
as measured in the space $L^2(0,1)$
with the standard (real) inner product and norm
\[
\langle g,h\rangle:=\int_0^1g(x)h(x)\,dx
\quad\text{and}\quad
\|g\|:=\langle g,g\rangle^{1/2}.
\]
Does $d_n\to0$ as $n\to\infty$? This apparently anodyne question takes on a
new significance
in view of the following remarkable result.

\begin{theorem}[Nyman \cite{Ny50}, B\'aez-Duarte \cite{BD03}]\label{T:NBD}
We have $\lim_{n\to\infty}d_n=0$  if and only if the Riemann hypothesis is true.
\end{theorem}

This result is really an amalgam of two theorems. 
We shall describe its history
in Section~\ref{S:NBD} and prove the ``only if'' part,
which is surprisingly easy.

The article \cite{LR02} describes the results of the computation of $d_n$ for $n$ up to 20000.
The authors of \cite{LR02} observe that
$d_n$ appears to be asymptotic to $C/\sqrt{\log n}$, where $C$ is a positive constant.
It is known that $d_n$ cannot decrease any faster than this \cite{BBLS00}.

While computing $d_n$ for ourselves,
we noticed a curious phenomenon.
Certain quantities arising in this computation always seem to be positive,
even though there is no obvious explanation as to why this should be.
This  leads us to formulate a conjecture,
one form of which is as follows.

\begin{conjecture}\label{Conj:det}
For all $j,k$ with $2\le j\le k$, we have the determinantal inequality
\[
\begin{vmatrix}
\langle f_2,f_2\rangle &\langle f_2,f_3\rangle &\dots &\langle f_2,f_{j-1}\rangle &\langle f_2,f_k\rangle\\
\langle f_3,f_2\rangle &\langle f_3,f_3\rangle &\dots &\langle f_3,f_{j-1}\rangle &\langle f_3,f_k\rangle\\
\vdots &\vdots&\vdots &\vdots &\vdots\\
\langle f_j,f_2\rangle &\langle f_j,f_3\rangle &\dots &\langle f_j,f_{j-1}\rangle &\langle f_j,f_k\rangle
\end{vmatrix}
>0.
\]
\end{conjecture}

Notice that the matrix in Conjecture~\ref{Conj:det} is not quite symmetric.
In the last column, where one might expect an $f_j$ there is instead an $f_k$.

We shall see that the inner products $\langle f_j,f_k\rangle$ are intimately related to
the Riemann zeta function.
Take a look, for example, at formula \eqref{E:fjfkzeta} at the end of the article. 
Indeed, in principle, the knowledge of these inner products suffices to determine whether
the Riemann hypothesis is true. 
There is thus a certain interest in understanding
them better, and, in particular, in explaining the phenomenon alluded to above.

The rest of the article is structured as follows.
In Section~\ref{S:NBD} we prove the ``only if'' part of Theorem~\ref{T:NBD} and briefly discuss
the ``if'' part. Then in Section~\ref{S:computation} we describe how to compute $d_n$.
In Section~\ref{S:conjecture}, we explain the background leading to the conjecture
and some first attempts to prove it.
There are also two appendices containing technical results needed about the functions $f_k$.
In Appendix~A we prove that the $f_k$ are linearly independent,
and in Appendix~B we derive some formulas for the inner products $\langle f_j,f_k\rangle$.


\section{The theorems of Nyman and B\'aez-Duarte.}\label{S:NBD}
We recall that the Riemann zeta function is defined by $\zeta(s):=\sum_{k\ge1}k^{-s}$ if $\Re s>1$,
and that it has a holomorphic extension to the whole of the punctured plane $\CC\setminus\{1\}$.
The Riemann hypothesis can be formulated as saying that $\zeta(s)$ has no zeros in the half-plane
$\Re s>1/2$.

The following theorem gives some information about 
the location of the zeros of $\zeta(s)$ in terms of the quantity
$d_n$  defined  in \eqref{E:dn}.

\begin{theorem}\label{T:zerofree}
$\zeta(s)$ has no zeros  in the disk $\Re s>(1+d_n|s|^2)/2$.
\end{theorem}

If $d_n\to0$ as $n\to\infty$, then the disks $\Re s>(1+d_n|s|^2)/2$ grow to fill the whole of the 
half-plane $\Re s>1/2$, and so we immediately deduce the following corollary.

\begin{corollary}\label{C:RH}
If $d_n\to0$, then the Riemann hypothesis is true.
\end{corollary}

The proof of Theorem~\ref{T:zerofree} is based on the following lemma,
which is essentially a computation of the  Mellin transform of $f_k$.

\begin{lemma}\label{L:Mellin}
For each $k\ge2$,
\begin{equation}\label{E:Mellin}
\int_0^1 f_k(x)x^{s-1}\,dx =
\frac{k^{-1}-k^{-s}}{s}\zeta(s) 
\quad(\Re s>0, ~s\ne1).
\end{equation}
\end{lemma}

\begin{proof}
Suppose first that $s$ is real and that $s>2$. We have
\begin{align*}
\int_0^1\Bigl[\frac{1}{x}\Bigr]x^{s-1}\,dx
&=\sum_{j=1}^\infty j \int_{1/(j+1)}^{1/j} x^{s-1}\,dx
=\frac{1}{s}\sum_{j=1}^\infty\Bigl(\frac{j}{j^s}-\frac{j+1}{(j+1)^s}+\frac{1}{(j+1)^s}\Bigr)\\
&=\frac{1}{s}\Bigl(1+\sum_{j=1}^\infty\frac{1}{(j+1)^s}\Bigr)
=\frac{1}{s}\sum_{j=1}^\infty \frac{1}{j^s}=\frac{\zeta(s)}{s}.
\end{align*}
Hence
\[
\int_0^1\Bigl[\frac{1}{kx}\Bigr]x^{s-1}\,dx
=k^{-s}\int_0^k\Bigl[\frac{1}{t}\Bigr]t^{s-1}\,dt
=k^{-s}\int_0^1\Bigl[\frac{1}{t}\Bigr]t^{s-1}\,dt
=\frac{k^{-s}}{s}\zeta(s).
\]
Thus \eqref{E:Mellin} holds for $s$ real with $s>2$. However,
since $f_k$ is a bounded function, the left-hand side of \eqref{E:Mellin} is a holomorphic
function of $s$ for $\Re s>0$.  The right-hand side of \eqref{E:Mellin} is also holomorphic 
in $\Re s>0$ (the zero of $(k^{-1}-k^{-s})$ at $s=1$ cancels the pole of $\zeta(s)$). 
By the identity principle for holomorphic functions, 
the equality \eqref{E:Mellin} persists in the whole half-plane.
\end{proof}

The following corollary of Lemma~\ref{L:Mellin}, though not needed in the proof of the main theorem,
will be useful later on.

\begin{corollary}\label{C:fk1}
For all $k\ge2$, we have 
$\langle f_k,1\rangle =(\log k)/k$.
\end{corollary}

\begin{proof}
By Lemma~\ref{L:Mellin}, together with the fact that $\zeta(s)=(s-1)^{-1}+O(1)$ as $s\to1$,
we obtain
\[
\int_0^1 f_k(x)x^{s-1}\,dx =
-\frac{k^{-s}-k^{-1}}{s-1}+o(1)
\quad(s\to 1).
\]
Letting $s\to1$ in this expression, we deduce that
\[
\langle f_k,1\rangle=-\frac{d}{ds}k^{-s}\Bigl|_{s=1}=\frac{\log k}{k}.\qedhere
\]
\end{proof}

\begin{proof}[Proof of Theorem~\ref{T:zerofree}]
Suppose that $\Re s>1/2$ and that $\zeta(s)=0$.
Certainly $s\ne1$, since $\zeta$ has a pole there.
Therefore, by Lemma~\ref{L:Mellin}, we have $\int_0^1 f_k(x)x^{s-1}\,dx=0$ for all $k$.
Consequently,
if $\lambda_2,\dots,\lambda_n\in\RR$, then
\[
\int_0^1\Bigl(1-\sum_{k=2}^n\lambda_kf_k(x)\Bigr)x^{s-1}\,dx=\int_0^1 x^{s-1}\,dx=\frac{1}{s}.
\]
On the other hand, by the Cauchy--Schwarz inequality,
\[
\Bigl |\int_0^1 \Bigl(1-\sum_{k=2}^n\lambda_kf_k(x)\Bigr) x^{s-1}\,dx\Bigr|^2
\le \int_0^1\Bigl(1-\sum_{k=2}^n \lambda_kf_k(x)\Bigr)^2\,dx\cdot\int_0^1 |x^{s-1}|^2\,dx.
\]
Since $\int_0^1 |x^{s-1}|^2\,dx=1/(2\Re s-1)$, we deduce that
\[
\int_0^1\Bigl(1-\sum_{k=2}^n \lambda_kf_k(x)\Bigr)^2\,dx\ge \frac{2\Re s-1}{|s|^2}
\quad(\lambda_2,\dots,\lambda_n\in\RR).
\]
Taking the minimum over all $\lambda_2,\dots,\lambda_n$, we obtain 
\[
d_n^2\ge (2\Re s-1)/|s|^2.
\]
Thus, if $s$ does not satisfy this last inequality, then $\zeta(s)\ne0$.
This proves the theorem.
\end{proof}

Corollary~\ref{C:RH} establishes the ``only if'' part of Theorem~\ref{T:NBD}.
It was proved by Nyman in his thesis \cite{Ny50} in 1950. He also established a weak form of the ``if'' part.
In fact he showed that the uncountable family of functions 
\[
f_\alpha(x):= \frac{1}{\alpha}\Bigl[\frac{1}{x}\Bigr]-\Bigl[\frac{1}{\alpha x}\Bigr] \quad(\alpha\in\RR,\alpha>1)
\]
spans a dense subspace of $L^2(0,1)$ if and only if the Riemann hypothesis is true.
He never published his result, but his  doctoral supervisor Beurling did publish a generalization
in \cite{Be55}, showing that, for $p\in(1,2]$, the $f_\alpha$ span a dense subspace of $L^p[0,1]$ 
if and only if $\zeta(s)$ has no zeros in $\Re s>1/p$. This approach to the Riemann hypothesis is often called the Nyman--Beurling criterion.

Somewhat later, in 1984, Bercovici and Foias \cite{BF84} proved that the $f_\alpha$ do indeed span a dense subspace of $L^1[0,1]$. Unfortunately, this tells us nothing new about the Riemann zeta function.

Much later still, in 2003, B\'aez-Duarte \cite{BD03} showed that, if the Riemann hypothesis is true, then $1$ lies in the $L^2(0,1)$-closure of the span of $\{f_2,f_3,f_4,\dots\}$. This is interesting because there is a closed formula for the inner products $\langle f_j,f_k\rangle$ (whereas no such formula is known for $\langle f_\alpha,f_\beta\rangle$ for general real $\alpha,\beta$). B\'aez-Duarte's proof is quite different from that of Nyman. 
While Nyman's approach is based on the general theory of translation-invariant subspaces of $L^2(\RR)$,
 B\'aez-Duarte's  depends on specific properties of the Riemann zeta function. The survey article of Bagchi \cite{Ba06} contains a very readable account of this.


\section{Computation of $\lowercase{d_n}$.}\label{S:computation}

We shall describe two ways of computing $d_n$.
Both methods take for granted the fact that the  functions $(f_k)_{k\ge2}$ are linearly
independent, which is proved in Appendix~A.
Both methods also assume that we know how to calculate the inner products 
$\langle 1,f_k\rangle$ and $\langle f_j,f_k\rangle$.
From Corollary~\ref{C:fk1}, we already know that $\langle 1,f_k\rangle=(\log k)/k$.
In Appendix~B, we derive a closed formula for the inner products $\langle f_j,f_k\rangle$.

The first way to compute $d_n$ is via  Gram's\footnote{Jorgen Gram (1850--1916) was a Danish mathematician,
whose name is now associated to Hilbert spaces and orthonormal sets. Interestingly, he also made important contributions to the study of the zeros of the Riemann zeta function, though his methods were quite different from those being discussed here.} formula.
Given $h_1,\dots,h_m\in L^2(0,1)$, we write
$G(h_1,h_2,\dots,h_m)$ for their \emph{Gramian}, namely, the determinant of the matrix of inner products
$(\langle h_j,h_k\rangle)_{j,k=1}^m$.
It satisfies
\[
G(h_1,h_2,\dots,h_m)\ge0,
\]
with equality if and only if the set $\{h_1,h_2,\dots,,h_m\}$ is linearly dependent.
For more background on this, we refer to \cite[Section~8.7]{Da75}.

\begin{proposition}\label{P:Gram}
\[
d_n^2=\frac{G(1,f_2,f_3,\dots,f_n)}{G(f_2,f_3,\dots,f_n)}.
\]
\end{proposition}

\begin{proof}
Let $f^*_n$ be the function in the span of $\{f_2,\dots,f_n\}$ that minimizes $\|1-f^*_n\|$.
As $f^*_n$ belongs to the span of $\{f_2,\dots,f_n\}$, simple row and column operations yield that
\begin{equation}\label{E:Gram1}
G(1,f_2,\dots,f_n)=G((1-f^*_n),f_2,\dots,f_n).
\end{equation}
Also $(1-f^*_n)$ is orthogonal to the span of $\{f_2,\dots,f_n\}$, 
and in particular we have $\langle 1-f^*_n,f_j\rangle=0$
for $j=2,\dots,n$. 
Therefore, all the off-diagonal entries in the first row and first column of 
$G((1-f^*_n),f_2,\dots,f_n)$ are zero, and hence
\begin{align}
G((1-f^*_n),f_2,\dots,f_n)&=\langle 1-f^*_n,1-f^*_n\rangle.G(f_2,\dots,f_n)\nonumber \\
&=d_n^2G(f_2,\dots,f_n).\label{E:Gram2}
\end{align}
Combining \eqref{E:Gram1} and \eqref{E:Gram2}, we obtain the result.
\end{proof}

In practice, the formula in Proposition~\ref{P:Gram} becomes unwieldy for large values of~$n$.
We shall now derive another method for computing $d_n$ which, 
though superficially more complicated, is computationally far superior.

Fix $n\ge 2$. By applying the Gram--Schmidt procedure to the sequence  $(f_k)_{2\le k\le n}$
we can obtain an orthonormal sequence of functions $(e_j)_{2\le j\le n}$ in $L^2(0,1)$ such that
\begin{equation}\label{E:span}
\spn\{e_2,\dots,e_k\}=\spn\{f_2,\dots,f_k\}\quad (2\le k\le n).
\end{equation}
The sequence $(e_j)$ if uniquely determined if we further normalize it so that
\begin{equation}\label{E:normalize}
\langle e_k,f_k\rangle>0 \quad (2\le k\le n).
\end{equation}
Henceforth, we  always assume this normalization.

By \eqref{E:span}, the quantity $d_n$ is just the $L^2$-distance of  $1$ from the span of $\{e_2,\dots,e_n\}$.
A simple computation  shows that, for any choice of scalars $\lambda_2,\dots,\lambda_n$, we have
\[
\Bigl\|1-\sum_{j=2}^n\lambda_j e_j\Bigr\|^2=1-\sum_{j=2}^n|\langle 1,e_j\rangle|^2+\sum_{j=2}^n|\lambda_j-\langle 1,e_j\rangle|^2,
\]
which is minimized by taking $\lambda_j:=\langle 1,e_j\rangle$ for all $j$. We thus obtain the formula
\begin{equation}\label{E:dnformula}
d_n^2=1-\sum_{j=2}^n|\langle 1,e_j\rangle|^2.
\end{equation}

It remains to express the inner products $\langle 1,e_j\rangle$ in terms of the data $\langle f_j,f_k\rangle$
and $\langle 1,f_k\rangle$. 
Define
\[
L_{kj}:=\langle f_k,e_j\rangle
\quad(j,k\in\{2,\dots,n\}).
\]
Note  that,  as $f_k\in\spn\{e_2,\dots,e_k\}$,
we have  $L_{kj}=\langle f_k,e_j\rangle=0$ if $j>k$. In other words, $L$ is a lower triangular matrix.
Further, $L_{kk}>0$ for all $k$ by \eqref{E:normalize}, so $L$ is invertible.
Expanding $f_k$ as $f_k=\sum_{j=2}^n \langle f_k,e_j\rangle e_j$, we have
\[
\langle 1,f_k\rangle=\sum_{j=2}^n\langle f_k,e_j\rangle\langle 1,e_j\rangle.
\]
In other words,  $F=LE$,  where $F_k:=\langle 1,f_k\rangle$ and $E_j:=\langle 1,e_j\rangle$.
Also, we have
\[
\langle f_k,f_l\rangle=\sum_{j=2}^n\langle f_k,e_j\rangle\langle e_j,f_l\rangle.
\]
In other words, $P=LL^t$, where $P_{kl}:=\langle f_k,f_l\rangle$. In fact, what we have done
is to construct the \emph{Cholesky decomposition} of $P$, namely the unique factorization of $P=LL^t$
where $L$ is lower triangular with positive entries on the diagonal.

We summarize these remarks in a proposition.

\begin{proposition}\label{P:GS}
Let $n\ge2$, and, for $j,k\in\{2,\dots,n\}$, define
\[
P_{jk}:=\langle f_j,f_k\rangle
\quad\text{and}\quad
F_k:=\langle 1,f_k\rangle=\frac{\log k}{k}.
\]
Let $P=LL^t$ be the Cholesky decomposition of $P$, and let $E$ be the solution of 
the triangular linear system $LE=F$.
Then
\[
d_n^2=1-\|E\|^2.
\]
\end{proposition}

There are efficient and stable numerical methods for computing the  Cholesky factorization
of a positive-definite matrix (see, e.g., \cite[Chapter~IV, Lecture~32]{TB97}). 
Thus Proposition~\ref{P:GS} is quite practical for large-scale computations.


\section{The Positivity Conjecture.}\label{S:conjecture} 
Tables~\ref{Tb:P} and~\ref{Tb:L} below show the first few entries of the matrices $P_{jk}:=\langle f_j,f_k\rangle$
and $L_{kj}:=\langle f_k,e_j\rangle$, rounded to four decimal places.

\begin{table}[ht]
\caption{Entries of $P_{jk}:=\langle f_j,f_k\rangle$ for $j,k\in\{2,\dots,9\}$.}
\begin{center}
\begin{tabular}{c|cccccccccc}\label{Tb:P}
 & 2 & 3 & 4 & 5 & 6 & 7 & 8 & 9 \\
 \hline
2 & 0.1733 & 0.1063 & 0.1184 & 0.0918 & 0.0931 & 0.0784 & 0.0778 & 0.0683\\				
3 & 0.1063 & 0.1770	& 0.1220 & 0.1118 & 0.1178 & 0.0976 & 0.0908 & 0.0914\\
4 & 0.1184 & 0.1220	& 0.1618 & 0.1194 & 0.1103 & 0.1023 & 0.1060	 & 0.0912\\	
5 & 0.0918 & 0.1118	& 0.1194 & 0.1456 & 0.1125 & 0.1019 & 0.0956 & 0.0918\\
6 & 0.0931 & 0.1178	& 0.1103 & 0.1125 & 0.1313 & 0.1049 & 0.0957	 & 0.0909\\	
7 & 0.0784 & 0.0976	& 0.1023 & 0.1019 & 0.1049 & 0.1192 & 0.0976 & 0.0889\\
8 & 0.0778 & 0.0908 & 0.1060 & 0.0956 & 0.0957 & 0.0976 & 0.1089 & 0.0910\\
9 & 0.0683 & 0.0914 & 0.0912 & 0.0918 & 0.0909 & 0.0889 & 0.0910 & 0.1002
\end{tabular}
\end{center}
\end{table}

\begin{table}[ht]
\caption{Entries of $L_{kj}:=\langle f_k,e_j\rangle$ for $j,k\in\{2,\dots,9\}$.}
\begin{center}
\begin{tabular}{c|cccccccccc}\label{Tb:L}
 & 2 & 3 & 4 & 5 & 6 & 7 & 8 & 9 \\
 \hline
2 &  0.4163 & 0 & 0 & 0 & 0 & 0 & 0 & 0\\ 
3 &0.2554 & 0.3343 & 0 & 0 & 0 & 0 & 0 & 0\\ 
4 & 0.2845 & 0.1475 & 0.2430 & 0 & 0 & 0 & 0 & 0\\ 
5 & 0.2205 & 0.1659 & 0.1325 & 0.2277 & 0 & 0 & 0 & 0\\ 
6 &  0.2237 & 0.1814 & 0.0819 & 0.0976 & 0.1792 & 0 & 0 & 0\\ 
7 &  0.1883 & 0.1480 & 0.1107 & 0.0929 & 0.0991 & 0.1764 & 0 & 0\\ 
8 & 0.1868 & 0.1288  & 0.1395 & 0.0638 & 0.0721 & 0.0841 & 0.1471 & 0\\ 
9 & 0.1641 & 0.1479 & 0.0934 & 0.0822 & 0.0651 & 0.0664 & 0.0863 & 0.1409
\end{tabular}
\end{center}
\end{table}

As expected, $P$ is symmetric and $L$ is lower triangular.
Another obvious feature is that all the entries are positive,
indeed strictly positive if one excludes the entries of $L$ above the diagonal.
Should we have expected this? In the case of $P$, the answer is certainly yes,
since, as remarked at the beginning of the article, 
$f_k(x)=1/k$ on $(\frac{1}{2},1]$ and is nonnegative elsewhere on $(0,1]$,  so
\[
\langle f_j,f_k\rangle\ge \frac{1}{2jk}>0.
\]
On the other hand, for $L$ the answer is not so clear. Certainly the diagonal elements of $L$
are positive, because of our normalization \eqref{E:normalize}. 
Also, the entries in the first column of $L$ are positive, because
\[
\langle f_k,e_2\rangle =\frac{\langle f_k,f_2\rangle}{\|f_2\|}>0
\quad(k\ge2).
\]
But as for the other entries strictly below the diagonal,
there does not seem to be any obvious explanation as to why they should be positive. 
Is it just a coincidence?

Intrigued by this, we have undertaken rather more detailed calculations, 
going much further than $n=9$ and to much higher precision. According to our 
computations,\footnote{The computations were performed using MATLAB 2017.}
\begin{equation}\label{E:50K}
\boxed{
\langle e_j,f_k\rangle>0
\quad\text{for all $j,k$ with $2\le j\le k\le 50000$}.
}
\end{equation}
This is certainly not typical of orthonormal sequences obtained via the Gram--Schmidt process.
Emboldened by the compelling numerical evidence, we make the following conjecture.

\begin{conjecture}\label{Conj:ip}
$\langle e_j,f_k\rangle>0$ for all $j,k$ with $2\le j\le k$.
\end{conjecture}

Though this looks simple, it involves the functions $e_j$, which are actually quite complicated to understand.
After all, the Riemann hypothesis itself is equivalent to the statement that $\sum_{j\ge2}|\langle 1,e_j\rangle|^2=1$
(just combine \eqref{E:dnformula} with Theorem~\ref{T:NBD}).
The following proposition gives some  criteria for the inequality $\langle e_j,f_k\rangle>0$,
expressed purely in terms of the original
functions $f_k$. It is convenient to introduce the notation
\[
G(f_2,\dots,f_i|g,h):=
\begin{vmatrix}
\langle f_2,f_2\rangle &\langle f_2,f_3\rangle &\dots &\langle f_2,f_i\rangle &\langle f_2,h\rangle\\
\langle f_3,f_2\rangle &\langle f_3,f_3\rangle &\dots &\langle f_3,f_i\rangle &\langle f_3,h\rangle\\
\vdots &\vdots&\vdots &\vdots &\vdots\\
\langle f_i,f_2\rangle &\langle f_i,f_3\rangle &\dots &\langle f_i,f_i\rangle &\langle f_i,h\rangle\\
\langle g,f_2\rangle &\langle g,f_3\rangle &\dots &\langle g,f_i\rangle &\langle g,h\rangle
\end{vmatrix}.
\]
If $i=1$, then we interpret $G(f_2,\dots,f_i|g,h)$ simply as $\langle g,h\rangle$.

\begin{proposition}\label{P:poscriteria}
Let $j\ge2$ and let $f\in L^2(0,1)$. The following statements are equivalent.
\begin{enumerate}
\item $\langle e_j,f\rangle>0$;
\item $G(f_2,\dots, f_{j-1}| f_j,f)>0$;
\item $G(f_2,\dots,f_{j-1},(f_j+f))>G(f_2,\dots,f_{j-1},(f_j-f))$.
\end{enumerate}
\end{proposition}

As a special case, we obtain the following corollary.

\begin{corollary}
$\langle e_j,f_k\rangle>0$ if and only if $G(f_2,\dots,f_{j-1}|f_j,f_k)>0$.
\end{corollary}

Thus Conjecture~\ref{Conj:ip} can be reformulated as Conjecture~\ref{Conj:det}.
One holds if and only the other does.

\begin{proof}[Proof of Proposition~\ref{P:poscriteria}]
Define a linear functional $\phi:L^2(0,1)\to\RR$ by
\[
\phi(h):=G(f_2,\dots,f_{j-1}|f_j,h).
\]
Clearly $\phi(h)=0$ if $h$ belongs to the span of $\{f_2,\dots,f_{j-1}\}$ or if $h$ is orthogonal to the span
of $\{f_2,\dots,f_j\}$. Given $f\in L^2(0,1)$, we can write it as
\[
f=\sum_{i=2}^{j-1}\langle f,e_i\rangle e_i+\langle f,e_j\rangle e_j +\Bigl(f-\sum_{i=2}^j\langle f,e_i\rangle e_i\Bigr).
\]
By the remarks just made,  $\phi$ vanishes on both the first term and the third. Hence
\[
\phi(f)=\langle f,e_j\rangle \phi(e_j).
\]
Note also that $\phi(f_j)=G(f_2,\dots,f_{j-1},f_j)>0$. It follows that $\phi(e_j)\ne0$ and that
\[
\frac{\langle e_j,f\rangle}{\langle e_j,f_j\rangle}=\frac{G(f_2,\dots,f_{j-1}|f_j,f)}{G(f_2,\dots,f_{j-1},f_j)}.
\]
Since both denominators are positive, we deduce that
\[
\langle e_j,f\rangle>0\iff G(f_2,\dots,f_{j-1}|f_j,f)>0,
\] 
which establishes the equivalence between parts (1) and (2).
  
The map $(g,h)\mapsto G(f_2,\dots,f_{j-1}|g,h)$ is a symmetric bilinear form.
By the polarization identity, it follows that
\[
4G(f_2,\dots,f_{j-1}|g,h)
=G(f_2,\dots,f_{j-1},(g+h))-G(f_2,\dots,f_{j-1},(g-h)).
\]
The equivalence between parts (2) and (3) follows immediately from this identity.
\end{proof}

This proposition enables us to quickly rule out a possible variant of the conjecture.
A calculation shows that
\[
G(f_6,f_3,f_4|f_5,f_2)\approx -1.6493\times 10^{-6}<0.
\]
In other words, if we exchange $f_2$ and $f_6$, then the conjecture no longer holds.
Thus the order of the $(f_k)$ matters.

Another application of the proposition was communicated to us by one of the anonymous referees,
whose contribution we gratefully acknowledge.  It is based on the following asymptotic formula
for the inner products $\langle f_j,f_k\rangle$: for each $j\ge2$,
\begin{equation}\label{E:asymp}
\langle f_j,f_k\rangle \sim \Bigl(\frac{j-1}{j}\Bigr)\Bigl(\frac{\log k}{2k}\Bigr) \quad(k\to\infty).
\end{equation}
A derivation of this formula is given at the end of Appendix~B. Feeding
the formula into the definition of $G(f_2,\dots,f_{j-1}|f_j,f_k)$,
we deduce that, for each $j\ge2$,
\[
\lim_{k\to\infty}\frac{G(f_2,\dots,f_{j-1}|f_j,f_k)}{(\log k)/2k}=H(j),
\]
where
\[
H(j):=
\begin{vmatrix}
\langle f_2,f_2\rangle &\langle f_2,f_3\rangle &\dots &\langle f_2,f_{j-1}\rangle &1/2\\
\langle f_3,f_2\rangle &\langle f_3,f_3\rangle &\dots &\langle f_3,f_{j-1}\rangle &2/3\\
\vdots &\vdots&\vdots &\vdots &\vdots\\
\langle f_j,f_2\rangle &\langle f_j,f_3\rangle &\dots &\langle f_j,f_{j-1}\rangle &(j-1)/j
\end{vmatrix}.
\]
In combination with Proposition~\ref{P:poscriteria}, this yields the following theorem.

\begin{theorem}\label{T:referee}
Let $j\ge2$. If $H(j)>0$, then there exists $k_0(j)$ such that $\langle e_j,f_k\rangle>0$ for all $k\ge k_0(j)$.
\end{theorem}

Our calculations show that $H(j)>0$ for all $j$ with $2\le j\le 100$.

According to the referee, it is possible to use a refinement of \eqref{E:asymp}
to obtain upper bounds for $k_0(j)$ for small $j$. 
Of course, once $k_0(j)$ is known for a given value of $j$,
then in principle one can verify directly that $\langle e_j,f_k\rangle>0$ for those $k$ between
$j$ and $k_0(j)$, and thereby establish that $\langle e_j,f_k\rangle>0$ for all $k\ge j$.

\section{Conclusion.}

Where does all this leave us? Is the conjecture true or not?
Does it imply the Riemann hypothesis?
Is it a consequence of the Riemann hypothesis?
Alas, we have been unable to answer any of these questions!
We offer them as a challenge to readers of this \textsc{Monthly}.


\section*{APPENDIX A: LINEAR INDEPENDENCE OF THE $\lowercase{f_k}$.}\label{S:lindep}

The methods for computing $d_n$ described in Section~\ref{S:computation}
take for granted the fact that the functions $(f_k)$ are linearly independent. 
We shall prove this linear independence by constructing a
biorthogonal sequence $(g_l)_{l\ge2}$, namely a sequence in $L^2(0,1)$
such that $\langle f_k,g_l\rangle=\delta_{kl}$ (where, as usual, $\delta_{kl}:=1$
if $k=l$ and $\delta_{kl}:=0$ otherwise). This construction is due to Vasyunin \cite{Va96}.

A key tool in the construction is the M\"obius function $\mu$. 
Recall that $\mu(n):=(-1)^r$ if $n$ is a product of $r$ distinct
prime numbers for some $r\ge0$, and $\mu(n):=0$ otherwise. 
The M\"obius function has the property  that
\begin{equation}\label{E:Mobius}
\sum_{j|n}\mu(j)=\delta_{1n}.
\end{equation}
For  background on the M\"obius function, we refer to \cite[Chapter~2]{Ap76}.

\begin{theorem}[Vasyunin \cite{Va96}]\label{T:biorthog}
For $l\ge2$, define $g_l\in L^2(0,1)$ by
\[
g_l:=\sum_{j|l}\mu(l/j)(h_{j-1}-h_{j}),
\]
where $h_j:=j(j+1)1_{(\frac{1}{j+1},\frac{1}{j}]}$ for $j\ge1$ and  $h_0:=0$.
Then
\[
\langle f_k,g_l\rangle=\delta_{kl} \quad(k,l\ge2).
\]
\end{theorem}

If $\sum_{k=2}^n\lambda_k f_k=0$, then,
taking the inner product with $g_l$ and using Theorem~\ref{T:biorthog}, we deduce that
$\lambda_l=0$ for each $l$.
Thus we obtain the following corollary.

\begin{corollary}
The functions $(f_k)_{k\ge2}$ are linearly independent.
\end{corollary}

Actually, more is true. Essentially the same argument shows that the sequence $(f_k)_{k\ge2}$ is 
\emph{minimal}, meaning that each $f_k$ is a positive distance  from the span of all the others.

\begin{proof}[Proof of Theorem~\ref{T:biorthog}]
We recall from Section~\ref{S:intro} that $f_k(x)=\{j/k\}$ on $(\frac{1}{j+1},\frac{1}{j}]$.
Hence
\[
\langle f_k,h_j\rangle=\int_{1/j+1}^{1/j}\Bigl\{\frac{j}{k}\Bigr\}j(j+1)\,dx=\Bigl\{\frac{j}{k}\Bigr\}.
\]
It follows that
\[
\langle f_k,g_l\rangle =\sum_{j|l}\mu(l/j)\Bigl(\Bigl\{\frac{j-1}{k}\Bigr\}-\Bigl\{\frac{j}{k}\Bigr\}\Bigr).
\]
Now
\[
\Bigl\{\frac{j-1}{k}\Bigr\}-\Bigl\{\frac{j}{k}\Bigr\}=
\begin{cases}
1-1/k, &\text{if $k$ divides $j$,}\\
\hfill -1/k, &\text{otherwise.}\\
\end{cases}
\]
Therefore,
\[
\langle f_k,g_l\rangle =-\frac{1}{k}\sum_{j|l}\mu(l/j)+\sum_{\substack{j|l\\ j\in k\NN}}\mu(l/j).
\]
Since $l\ge2$, the first of these sums vanishes by \eqref{E:Mobius}.
As for the second sum, it contains no terms unless $k$ divides $l$.
If $k$ divides $l$, then, writing $l=kl'$ and $j=kj'$, we have
\[
\sum_{\substack{j|l\\ j\in k\NN}}\mu(l/j)=\sum_{j'|l'}\mu(l'/j')=\delta_{1l'}=\delta_{kl}.
\]
This completes the proof of the theorem.
\end{proof}


\section*{APPENDIX B: FORMULAS FOR ${\langle\lowercase{ f_j,f_k}\rangle}$.}\label{S:ipformula}

To apply the  methods described in Section~\ref{S:computation},
we need to compute the inner products $\langle f_j,f_k\rangle$.
In this section we derive some formulas for these inner products.

As we have already remarked, the functions $f_k$ are all constant on intervals of the form
$(\frac{1}{r+1},\frac{1}{r}]$, where $r$ is a positive integer.
Thus $\langle f_j,f_k\rangle$ can be expressed as the sum of an infinite series.
But actually more is true. Looking again at the graph of $f_5$ in Figure~\ref{F:f5}, 
we see that it ``repeats itself.''
The same is true of every $f_k$. This periodicity property can be exploited to re-express $\langle f_j,f_k\rangle$
as the sum of a \emph{finite} series.
This idea seems to have been first noticed by Vasyunin \cite{Va96}, 
who used it to give several formulas
for $\langle f_j,f_k\rangle$. 
We shall derive one such formula, 
from which the others can  be deduced.

\begin{theorem}[Vasyunin \cite{Va96}]\label{T:ipformula}
Let $j,k\ge2$, let $m$ be a common multiple of $j,k$, and let $\omega:=\exp(2\pi i/m)$. Then
\begin{equation}\label{E:ipformula}
\langle f_j,f_k\rangle=
\frac{1}{m}\sum_{q=1}^{m-1}\sum_{r=1}^{m-1}
\Bigl\{\frac{q}{j}\Bigr\}\Bigl\{\frac{q}{k}\Bigr\}\omega^{-qr}(\omega^{-r}-1)\log(1-\omega^r).
\end{equation}
\end{theorem}

\begin{proof}
For each positive integer $r$ and each $x\in(\frac{1}{r+1},\frac{1}{r}]$, we have
$f_j(x)=\{r/j\}$ and $f_k(x)=\{r/k\}$. Therefore,
\[
\int_0^1 f_j(x)f_k(x)\,dx
=\sum_{r=1}^\infty\int_{1/(r+1)}^{1/r} f_j(x)f_k(x)\,dx
=\sum_{r=1}^\infty\Bigl\{\frac{r}{j}\Bigr\}\Bigl\{\frac{r}{k}\Bigr\}\frac{1}{r(r+1)}.
\]
As $m$ is a common multiple of $j$ and $k$, 
the functions $r\mapsto \{r/j\}$ and $r\mapsto\{r/k\}$ are $m$-periodic
and vanish at multiples of $m$. Therefore,
\begin{equation}\label{E:midway}
\sum_{r=1}^\infty\Bigl\{\frac{r}{j}\Bigr\}\Bigl\{\frac{r}{k}\Bigr\}\frac{1}{r(r+1)}
=\sum_{q=1}^{m-1} \Bigl\{\frac{q}{j}\Bigr\}\Bigl\{\frac{q}{k}\Bigr\}\Bigl( \sum_{\substack{r=1\\ r\equiv q\text{\,mod\,}m}}^\infty\frac{1}{r(r+1)}\Bigr).
\end{equation}
We seek to express the last series (in brackets) in finite terms.
To this end, we note that
\[
\sum_{r=1}^\infty \frac{z^r}{r(r+1)}=1+\frac{(1-z)\log(1-z)}{z} \quad(|z|<1).
\]
In fact, the equality holds even if $|z|=1$, provided 
that one interprets the product $(1-z)\log(1-z)$ as being $0$ when $z=1$.
In particular, the equality holds for $z=\omega^l$, where $\omega:=\exp(2\pi i/m)$. Therefore,
\[
\sum_{r=1}^\infty \frac{\omega^{rl}}{r(r+1)}=1+\frac{(1-\omega^l)\log(1-\omega^l)}{\omega^l}.
\]
Multiplying  both sides by $\omega^{-ql}$ and summing from $l=1$ to $l=m$, 
we get
\[
\sum_{l=1}^m\sum_{r=1}^\infty\frac{\omega^{(r-q)l}}{r(r+1)}
=\sum_{l=1}^m\omega^{-ql}\Bigl(1+\frac{(1-\omega^l)\log(1-\omega^l)}{\omega^l}\Bigr).
\]
Now, as is well known,
\[
\frac{1}{m}\sum_{l=1}^m\omega^{(r-q)l}=
\begin{cases} 1, &\text{if~}m|(r-q),\\ 
0, &\text{otherwise}.\end{cases}
\]
It follows that
\begin{align*}
\sum_{\substack{r= 1\\r\equiv q\text{\,mod\,}m}}^\infty\frac{1}{r(r+1)}
&=\frac{1}{m}\sum_{l=1}^m\omega^{-ql}\Bigl(1+\frac{(1-\omega^l)\log(1-\omega^l)}{\omega^l}\Bigr)\\
&=\delta_{\{m|q\}}+\frac{1}{m}\sum_{l=1}^{m-1}\omega^{-ql}(\omega^{-l}-1)\log(1-\omega^l),
\end{align*}
where $\delta_{\{m|q\}}$ is $1$ if $q$ is divisible by $m$, and  is $0$ otherwise.
Substituting this back into \eqref{E:midway}, 
the $\delta_{\{m|q\}}$ term disappears, and we obtain \eqref{E:ipformula}.
\end{proof}

This is not the end of the story. The double sum in  \eqref{E:ipformula} can be developed still further,
into a single sum involving only real functions. We content ourselves to state the end result, 
referring to Vasyunin's paper  \cite{Va96} for the details, which are elementary but quite long. 
Let $d$ be the greatest common divisor of $j,k$, and write
$j=dj_0$ and $k=dk_0$. 
Thus $j_0,k_0$ are coprime, so there exist integers $a,b$ such that $aj_0+bk_0=1$. Then,
according to Vasyunin's formula:
\begin{equation}\label{E:Vasyunin}
\begin{aligned}
jk\langle f_j,f_k&\rangle
=\Bigl(\frac{k-1}{2}\Bigr)\log j+\Bigl(\frac{j-1}{2}\Bigr)\log k\\
&-\frac{\pi}{2}\sum_{r=1}^{j-1}\Bigl(\frac{1}{2}-\frac{r}{j}\Bigr)\cot\Bigl(\frac{\pi r}{j}\Bigr)
-\frac{\pi}{2}\sum_{r=1}^{k-1}\Bigl(\frac{1}{2}-\frac{r}{k}\Bigr)\cot\Bigl(\frac{\pi r}{k}\Bigr)\\
&+\frac{\pi d}{2}\sum_{r=1}^{j_0-1}\Bigl(\frac{1}{2}-\frac{r}{j_0}\Bigr)\cot\Bigl(\frac{\pi rb}{j_0}\Bigr)
+\frac{\pi d}{2}\sum_{r=1}^{k_0-1}\Bigl(\frac{1}{2}-\frac{r}{k_0}\Bigr)\cot\Bigl(\frac{\pi ra}{k_0}\Bigr).
\end{aligned}
\end{equation}
In fact, this is the formula that we used in the computations leading to \eqref{E:50K}.

Cotangent sums of the type above were recently 
studied by Bettin and Conrey \cite{BC13}. They showed that these sums exhibit
a type of reciprocity property. We refer to their paper  for the details.

Next, we derive another formula for $\langle f_j,f_k\rangle$,
which, though less useful from the point of view of computation,
exhibits a certain structure that may eventually help in proving Conjecture~\ref{Conj:det}.
It also brings out the relationship between the inner products $\langle f_j,f_k\rangle$ 
and the Riemann zeta function explicitly.
The formula is closely related to Lemma~\ref{L:Mellin}, and was very likely known to Nyman and Beurling.

\begin{theorem}
For $j,k\ge2$, we have
\begin{equation}\label{E:fjfkzeta}
\langle f_j, f_k\rangle=
\frac{1}{2\pi jk}\int_{-\infty}^\infty \bigl(j^{\frac{1}{2}-it}-1\bigr)\bigl(k^{\frac{1}{2}+it}-1\bigr)\frac{\bigl|\zeta(\frac{1}{2}+it)\bigr|^2}{\frac{1}{4}+t^2}\,dt.
\end{equation}
\end{theorem}

\begin{proof}
For each $k\ge2,$ define $h_k:\RR\to\RR$ by
\[
h_k(t):=
\begin{cases}
f_k(e^{-t})e^{-t/2}, &t\ge0,\\
0, &t<0.\\
\end{cases}
\]
As $f_k$ is a bounded function, it follows that $h_k\in L^1(\RR)\cap L^2(\RR)$. 
The Fourier transform of $h_k$ can be computed as
\[
\widehat{h}_k(\omega)
=\int_{-\infty}^\infty h_k(t)e^{-i\omega t}\,dt
=\int_0^1 f_k(x)x^{-\frac{1}{2}+i\omega}\,dx
=\frac{k^{-1}-k^{-\frac{1}{2}-i\omega}}{\frac{1}{2}+i\omega}\zeta\Bigl(\frac{1}{2}+i\omega\Bigr),
\]
the last equality coming from \eqref{E:Mellin}. 
By Plancherel's theorem, we have
\[
\int_{-\infty}^\infty h_j(t)\overline{h_k(t)}\,dt
=\frac{1}{2\pi}\int_{-\infty}^\infty \widehat{h}_j(\omega)\overline{\widehat{h}_k(\omega)}\,d\omega.
\]
After the change of variable $x=e^{-t}$, the left-hand side is just $\langle f_j,f_k\rangle$. 
As for the right-hand side, it is equal to 
\[
\frac{1}{2\pi}\int_{-\infty}^\infty \frac{(j^{-1}-j^{-\frac{1}{2}-i\omega})(k^{-1}-k^{-\frac{1}{2}+i\omega})}{(\frac{1}{2}+i\omega)(\frac{1}{2}-i\omega)}\zeta\Bigl(\frac{1}{2}+i\omega\Bigr)\overline{\zeta\Bigl(\frac{1}{2}+i\omega\Bigr)}\,d\omega,
\]
which leads to \eqref{E:fjfkzeta}.
\end{proof}

Finally, we derive the asymptotic formula \eqref{E:asymp} for $\langle f_j,f_k\rangle$.
It is based on  properties of the function
\[
A(\lambda):=\int_0^\infty \{t\}\{\lambda t\}\,\frac{dt}{t^2}
\quad(\lambda>0),
\]
which was studied in detail by
B\'aez-Duarte, Balazard, Landreau, and Saias in \cite{BBLS05}.
This function exhibits some remarkable behavior.
For example, even though it is continuous, it has a strict local maximum at each rational number.
The property that we shall need,
established in \cite[Proposition 1]{BBLS05}, is that
\begin{equation}\label{E:asympA}
A(\lambda)\sim\frac{1}{2}\log\lambda \quad(\lambda\to\infty).
\end{equation}

\begin{theorem}\label{T:asymp}
For each $j\ge2$, we have
\[
\langle f_j,f_k\rangle \sim \frac{j-1}{j}\frac{\log k}{2k} \quad(k\to\infty).
\]
\end{theorem}

\begin{proof}
Rewriting $f_k$ as 
\[
f_k(x)= -\frac{1}{k}\Bigl\{\frac{1}{x}\Bigr\}+\Bigl\{\frac{1}{kx}\Bigr\}\qquad(0<x\le 1),
\]
and making the substitution $t:=1/x$, we see that
\begin{equation}\label{E:fjfkA}
\langle f_j,f_k\rangle=
\frac{1}{k}A\Bigl(\frac{k}{j}\Bigr)-\frac{1}{jk}A(k)-\frac{1}{jk}A(j)+\frac{1}{jk}A(1).
\end{equation}
The result follows upon combining \eqref{E:fjfkA} and \eqref{E:asympA}.
\end{proof}


\begin{acknowledgment}{Acknowledgments.}
The authors thank Andr\'e Fortin for his valuable advice concerning the computations
leading to \eqref{E:50K}.
They also thank the anonymous referees  for their 
careful reading of the paper and for 
suggestions that greatly improved the paper. 
HB was supported by an NSERC undergraduate student research award.
TR was supported by grants from NSERC and the Canada Research Chairs program.
\end{acknowledgment}



\begin{biog}

\item[Hugues Bellemare] 
is an undergraduate student in Mathematics at Universit\'e Laval. He has always been curious about the Millennium Problems. For now, his research interests are not very well defined, varying from analysis to algebra.
\begin{affil}
D\'epartement de math\'ematiques et de statistique, Universit\'e Laval, Qu\'ebec (QC), Canada G1V 0A6\\
hugues.bellemare.1@ulaval.ca
\end{affil}

\item[Yves Langlois] 
received his B.Sc.\ in Mathematics and  M.Sc.\ in Financial Engineering from  Universit\'e Laval. Since 2007 he has been working as a business intelligence analyst for the Government of Qu\'ebec, as well as studying computer science, travelling, and enjoying life with his family.
\begin{affil}
D\'epartement de math\'ematiques et de statistique, Universit\'e Laval, Qu\'ebec (QC), Canada G1V 0A6\\
yves.langlois.1@ulaval.ca
\end{affil}

\item[Thomas Ransford]  
received his Ph.D.\ from  the University of Cambridge. After spells as a university lecturer at Leeds and at Cambridge, 
he moved to Universit\'e Laval, where he now holds the Canada Research Chair in Spectral Theory and Complex Analysis.

\begin{affil}
D\'epartement de math\'ematiques et de statistique, Universit\'e Laval, Qu\'ebec (QC), Canada G1V 0A6\\
thomas.ransford@mat.ulaval.ca
\end{affil}

\end{biog}

\end{document}